\documentclass[a4paper]{amsart}

\usepackage{amscd,amsthm,amsfonts,amssymb,amsmath}
\usepackage[colorlinks=true]{hyperref}
\usepackage{fullpage}
\usepackage[all]{xy}
\usepackage{mathrsfs}
\usepackage{bm}

\newcommand{\BA}{{\mathbb {A}}}

\newcommand{\BC}{{\mathbb {C}}}

\newcommand{\BF}{{\mathbb {F}}}

\newcommand{\BN}{{\mathbb {N}}}

\newcommand{\BP}{{\mathbb {P}}}
\newcommand{\BQ}{{\mathbb {Q}}}

\newcommand{\BS}{{\mathbb {S}}}

\newcommand{\BZ}{{\mathbb {Z}}}

\newcommand{\CA}{{\mathcal {A}}}

\newcommand{\CH}{{\mathcal {H}}}

\newcommand{\CN}{{\mathcal {N}}}
\newcommand{\CO}{{\mathcal {O}}}

\newcommand{\CR }{{\mathcal {R}}}

\newcommand{\fa}{{\mathfrak{a}}}

    \newcommand{\fm}{{\mathfrak{m}}} 
     \newcommand{\fp}{{\mathfrak{p}}}

     \newcommand{\fN}{{\mathfrak{N}}}

\newcommand{\msa}{\mathscr{A}}
\newcommand{\msh}{\mathscr{H}}

\newcommand{\RT}{{\mathrm {T}}}

\newcommand{\alg}{{\mathrm{alg}}}

\newcommand{\Aut}{{\mathrm{Aut}}}

\newcommand{\End}{{\mathrm{End}}}

\newcommand{\Frob}{{\mathrm{Frob}}}

\newcommand{\Gal}{{\mathrm{Gal}}}
\newcommand{\GL}{{\mathrm{GL}}}

\newcommand{\Hom}{{\mathrm{Hom}}}

\newcommand{\Isom}{{\mathrm{Isom}}}

\newcommand{\Nr}{{\mathrm{Nr}}}

\newcommand{\ord}{{\mathrm{ord}}}

\newcommand{\rank}{{\mathrm{rank}}}

\newcommand{\Pic}{\mathrm{Pic}}

 \font\cyr=wncyr10

    \newcommand{\Sha}{\hbox{\cyr X}}

\newcommand{\SL}{{\mathrm{SL}}}

\newcommand{\tor}{{\mathrm{tor}}}

\newcommand{\Vol}{{\mathrm{Vol}}}

\newcommand{\ab}{{\mathrm{ab}}}

\renewcommand{\mod}{\, \mathrm{mod}\, }
\renewcommand{\pmod}[1]{\, (\mathrm{mod}\, #1) }

\newcommand{\wh}{\widehat}

\newcommand{\pair}[1]{\langle {#1} \rangle}

\newcommand{\ov}{\overline}

\newcommand{\lra}{\longrightarrow}
\newcommand{\lla}{\longleftarrow}
\newcommand{\ra}{\rightarrow}

\newcommand{\bs}{\backslash}

\newtheorem{thm}{Theorem}[section]

\newtheorem{lem}[thm]{Lemma}
\newtheorem{prop}[thm]{Proposition}

\newtheorem{defn}[thm]{Definition}

\newtheorem{theorem}{Theorem}[section]

\theoremstyle{definition}

\newtheorem{example}[theorem]{Example}

\theoremstyle{remark}
\newtheorem{remark}[thm]{Remark}

\numberwithin{equation}{subsection}
\newcommand{\sk}{\medskip}
\newcommand{\s}{\sk\noindent}

    \numberwithin{equation}{section}
\def\mat(#1,#2,#3,#4){
  \begin{pmatrix}
  #1 & #2 \\ #3 & #4
  \end{pmatrix}
}

\begin{document}
%------------------------------------------------------
\title{Heegner points on modular curves}

\author{Li Cai, Yihua Chen and Yu Liu}

\address{Li Cai: Yau Mathematical Sciences Center, Tsinghua University, Beijing
100084} \email{lcai@math.tsinghua.edu.cn}

\address{Yihua Chen: Academy of Mathematics and Systems
Science, Morningside center of Mathematics, Chinese Academy of
Sciences, Beijing 100190} \email{yihuachenamss@163.com}

\address{Yu Liu: Yau Mathematical Sciences Center, Tsinghua University, Beijing
100084} \email{liuyumaths@163.com}

\thanks{Li Cai was supported by the Special Financial Grant from the China Postdoctoral Science Foundation 2014T70067.}

\begin{abstract}
In this paper, we study the Heegner points on more general modular curves other than $X_0(N)$,
which generalizes Gross' work ''Heegner points on $X_0(N)$".
The explicit Gross-Zagier formula and the Euler system property are stated in this case.
Using such kind of Heegner points, we construct certain families of quadratic twists of $X_0(36)$,
with the ranks of Mordell-Weil groups being zero and one respectively,
and show that the $2$-part of their BSD conjectures hold.
\end{abstract}

\maketitle

\tableofcontents

\section{\bf Introduction}
Let $\phi=\sum_{n=1}^\infty a_n q^n$ be a newform of weight $2$, level $\Gamma_0(N)$, normalized such that $a_1=1$. Let $K$ be an imaginary quadratic field of discriminant $D$ and $\chi$ a (primitive) ring class character over $K$ of conductor $c$, i.e. a character of $\Pic(\CO_c)$ where $\CO_c$ is the order $\BZ+c\CO_K$ of $K$.  Let $L(s,\phi,\chi)$ be the Rankin-Selberg convolution of $\phi$ and $\chi$.
Assume the Heegner condition:
\begin{enumerate}
\item  $(c, N)=1$,
\item   Any prime $p|N$ is either split in $K$ or ramified in $K$ with $\ord_p(N)=1$
	and $\chi([\fp])\neq a_p$, where $\fp$ is the unique prime ideal of $\CO_K$ above $p$
	and $[\fp]$  is its class in $\Pic(\CO_c)$.
\end{enumerate}

Under this condition, the sign of $L(s,\phi,\chi)$ is $-1$ and
Gross studies the Heegner points on $X_0(N)$ in \cite{GR2}.
It's well known that $X_0(N)(\BC)$ parameterizes the pairs $(E,C)$, with $E$ an elliptic curve over $\BC$ and $C$ a cyclic subgroup of $E$ of order $N$. By the Heegner condition, there exists a proper ideal $\CN$ of $\CO_c$ such that $\CO_c/\CN \cong \BZ/N\BZ$. For any proper ideal $\fa$ of $\CO_c$, let $P_\fa\in X_0(N)$ be the point representing $(\BC/\fa,\fa\CN^{-1}/\fa)$, which is defined over the ring class field $H_c$, the abelian extension of $K$ with Galois group $\Pic(\CO_c)$ by class field theory. Such points are called Heegner points over $K$ of conductor $c$ and only depends on the class of $\fa$ in $\Pic(\CO_c)$.

Let $J_0(N)$ be the Jacobian of $X_0(N)$ and the cusp $[\infty]$ on $X_0(N)$ defines a morphism from $X_0(N)$ to $J_0(N)$ over $\BQ$: $P\mapsto [P-\infty]$. Let $P_\chi$ be the point $$P_\chi=\sum_{[\fa]\in \Pic(\CO_c)} [P_\fa-\infty]\otimes \chi([\fa])\in J_0(N)(H_c)\otimes_\BZ \BC$$
and $P_\chi^\phi$  the  $\phi$-isotypical component of $P_\chi$. Then under the Heegner condition, Cai-Shu-Tian \cite{CST} gives an explicit form of Gross-Zagier formula which relates height of $P_\chi^\phi$ to $L'(1,\phi,\chi)$. In fact, they give an explicit form of Gross-Zagier formula in general Shimura curve case.

Let the data $(\phi,K,\chi)$ be as above, and generalize the Heegner condition to the following one $(*)$:
\begin{enumerate}
\item[(i)]  $(c, N)=1$,
\item[(ii)]  if prime $p|N$ is inert in $K$, then $\ord_pN$ is even; if $p|N$ is ramified in $K$, then $\ord_pN=1$ and $\chi([\fp])\neq a_p$, where $\fp$ is the unique prime ideal of $\CO_K$ above $p$ and $[\fp]$  is its class in $\Pic(\CO_c)$.
\end{enumerate}

By this assumption, write $N=N_0N_1^2$, with $p|N_1$ if and only if $p$ is inert. Given an embedding $K\hookrightarrow M_2(\BQ)$ such that $K\cap M_2(\BZ)=K\cap R_0(N_0)=\CO_c$, where
$$R_0(N_0)=\left\{A\in M_2(\BZ)\Big|A\equiv\mat(*,*,0,*)\pmod {N_0}\right\}.$$
Then $R=\CO_c+N_1R_0(N_0)$ is an order of $M_2(\BZ)$. Define
$$\Gamma_K(N)=R^\times\cap\SL_2(\BZ)=\left\{ A\in\SL_2(\BZ)\left|\begin{array}{l}
                                        A\equiv\mat(*,*,0,*)\pmod{N_0}\\
                                        A\mod N_1\in R/N_1R
                                     \end{array}\right.
\right\}.$$
The modular curve now we have to consider is $X_K(N)=\Gamma_K(N)\bs\msh\cup\{\text{cusps}\}$.

This modular curve is not the usual modular curve of the form $X_0(M)$ any longer, if $N_1\neq 1$. $X_K(N)$ parameterizes $(E,C,\alpha)$ where $E$ is an elliptic curve over $\BC$, $C$ is a cyclic subgroup of $E$ of order $N_0$ and $\alpha$ is an $H$-orbit of an isomorphism $(\BZ/N_1\BZ)^2\simeq E[N_1]$, where $$H=(\CO_K/N_1\CO_K)^\times\subset\GL_2(\BZ/N_1\BZ).$$
The readers are referred to \cite{KM}. Let $h_0$ be the fixed point of $\msh$ under the action of $K^\times$.
Note that $\BZ+\BZ h_0^{-1}$ is an invertible ideal of $\CO_c$, then the
triple $P=\left(\BC/\BZ+\BZ h_0,\big\langle\frac{1}{N_0}\big\rangle,H\left({\frac{h_0}{N_1}\atop \frac{1}{N_1}}\right)\right)$ is a Heegner point on $X_K(N)$.
By CM theory, $P\in X_K(N)(K^\ab)$. The conductor of $P$ is also defined to be $c$, the conductor of $\CO_c$. For details, see Section 2.

Assume $\phi$ corresponds to an elliptic curve $E/\BQ$,
then there is a modular parametrization $f:X_K(N)\to E$, taking $\infty$ to identity in $E$.
It is unique in the sense that given two parametrizations $f_1,f_2$, there exist integers $n_1,n_2$ such that $n_1f_2=n_2f_2$
\cite[Proposition 3.8]{CST}. Now we can formulate the following Gross-Zagier formula:
\begin{thm}[\cite{CST}]
Under the assumption $(*)$
$$L'(1,E,\chi)=2^{-\mu(N,D)}\cdot\frac{8\pi^2(\phi,\phi)_{\Gamma_0(N)}}{u^2c\sqrt{|D_K|}}\frac{\wh{h}_K(P_\chi(f))}{\deg f}\eqno(\rm{GZ})$$
where $(\ ,\ )_{\Gamma_0(N)}$ is the Petersson inner product,
$\hat{h}_K$ is the N\'{e}ron-Tate height over $K$, $\mu(N,D)$ is the number of prime factors of $(N,D)$, $u=[\CO_c^\times:\{\pm 1\}]$.
\end{thm}

In another way, following the idea of Kolyvagin, these Heegner points form an ``Euler system".
There is a norm compatible relation between Heegner points of different conductors. See Theorem \ref{Nr:He}.

As an application of such kind of parametrization,
we will construct a family of quadratic twists of an elliptic curve with Mordell-Weil groups being  rank one.
The action of complex conjugation on the CM-points of modular curve is a crucial point in the proof of the nontriviality
of the Heegner point. For usual modular curve $X_0(N)$, the complex conjugation is essentially the Atkin-Lehner operator.
However, it does not keep for the modular curve $X_K(N)$.
We find the correct one, namely, the combination of local Atkin-Lehner operators? and the nontrivial normalizer of $K^\times$ in $\GL_2(\BQ)$.
Denote this operator by $w$. Then $f+f^w$ is a constant map with its image a nontrivial $2$-torsion point; see
Lemma \ref{key-lemma}. This phenomenon and the norm compatible relation control the divisibility of Heegner cycles.
Together with the Gross-Zagier formula,
the divisibility of Heegner cycles deduces the $2$-part of the BSD conjecture for our family of quadratic twists.

%\vspace{3mm}
For each square free non-zero integer $d \neq 1$, we write $E^{(d)}$ for the twist of an elliptic curve $E/\BQ$ by the quadratic extension $\BQ(\sqrt{d})/\BQ$. The results of \cite{Waldspurger}, \cite{BFH}, \cite{FH}, \cite{RKM} shows that there are infinitely many $d$ such that $L(E^{(d)}, s)$ is nonvanishing at $s=1$, and infinitely many $d$ such that $L(E^{(d)},s)$ has a simple zero at $s=1$.

The work of \cite{Tian}, \cite{Tian2} for the elliptic curve $(X_0(32),[\infty]): y^2 = x^3 - x$ constructs explicitly families of $d$
with $\ord_{s=1}L(E^{(d)},s) = 1$.
The work of \cite{CLTZ} deals with the elliptic curve $E=(X_0(49),[\infty])$ which has CM by $\sqrt{-7}$,
gets the similar results to \cite{Tian}.

Here, we construct a family of quadratic twists of $E=(X_0(36),[\infty])$ such that the ranks of Mordell-Weil groups for
this twists are one.
%We prove this theorem by construct Heegner points on our modular parametrization, and prove its non-triviality. More precisely, we study the Euler system formed by Heegner points of different conductors, which help us to compute its 2-divisibility. Then applying Gross-Zagier formula, we relates the Heegner points to L-function, and the $2$-divisibility tells us the information on BSD conjecture. Our theorem is:

\begin{thm}\label{th:one}
Let $\ell$ be a prime such that $3$ is split in $\BQ(\sqrt{-\ell})$ and $2$ is unramified in $\BQ(\sqrt{-\ell})$.
Let $M=q_1\cdots q_r$ be a positive square-free integer with prime factors $q_i$ all inert in $\BQ(\sqrt{-3})$, $q_i\equiv 1\pmod 4$
and $q_i$ inert in $\BQ(\sqrt{-\ell})$. Then
\begin{enumerate}
  \item $\ord_{s=1}L(s,E^{(-\ell M)})=1=\rank E^{(-\ell M)}(\BQ)$;
  \item $\#\Sha(E^{(-\ell M)}/\BQ)$ is odd, and the $p$-part of the full BSD conjecture of $E^{(-\ell M)}$ holds for $p\nmid 3\ell M$.
\end{enumerate}
\end{thm}

The nontriviality of Heegner cycles and Gross-Zagier formula also implies the rank part of BSD conjecture, for $E^{(M)}$, namely,
$$\ord_{s=1}L(s,E^{(M)})=0=\rank E^{(M)}(\BQ).$$
However, the proof of the $2$-part of the BSD conjecture for $E^{(-\ell M)}$ needs that for $E^{(M)}$.

A new feature of this paper is that we give a parallel proof of the BSD conjecture for $E^{(M)}$, that is, using
the Waldspurger formula and the norm property of Gross points. For the induction method using in
\cite{CLTZ}, \cite{TYZ}, there is an embedding problem of imaginary quadratic fields to quaternion algebras
which relating with the problem of representation  integers by ternary quadratic forms (See also
the argument before \cite[Definition 5.5]{CLTZ}, \cite[Section 2.1]{TYZ} and \cite{Q}). The use of the norm property
of Gross points avoids this emebdding problem.

%{\color{red}{A new feature of this paper is that we give a parallel proof of the BSD conjecture for $E^{(M)}$, that is, using
%the Waldspurger formula and the norm property of Gross points. The use of norm property
%of Gross points avoids the emebdding problem that \cite{CLTZ}, \cite{TYZ} need for the induction method. There, the embedding problem of imaginary quadratic fields to quaternion algebras
%which relating with the problem of representation  integers by ternary quadratic forms (See also
%the argument before \cite[Definition 5.5]{CLTZ}, \cite[Section 2.1]{TYZ} and \cite{Q}).}}

%Let $B$ be a definite quaternion algebra, $R$ an order of $B$ and $K$ be an imaginary quadratic field.
%Then a \textit{\textbf{Gross point}}
%of conductor $c$ is an element $(\xi,g)$ in $B^\times\bs\Hom(K,B)\times\wh{B}^\times/\wh{R}^\times$ such that $\xi^{-1}(g^{-1}Rg)=\CO_c$.

If the data $(\phi,K,\chi)$
satisfies that the root number $\epsilon(\phi,\chi)=+1$, by \cite{CST} we can choose an appropriate
definite quaternion algebra $B$ over $\BQ$
containing $K$, an order $R$ of $B$ of discriminant $N$ with $R\cap K = \CO_K$
and a ``unique'' function $f:B^\times\bs\wh{B}^\times/\wh{R}^\times\to \BC$.
Assume the conductor $c$ of $\chi$ satisfies $(c,N)=1$. Let $x_c \in K^\times \bs \wh{B}^\times / \wh{R}^\times$ be a Gross point
of conductor $c$, that is $x_c \wh{R} x_c^{-1} \cap \wh{K} = \wh{\CO}_c$. Denote by
\[P_\chi(f) = \sum_{\sigma \in \Gal(H_c/K)} f(\sigma.x_c)\chi(\sigma).\]
Then with similar notations as for Gross-Zagier formula, we have the Waldspurger formula (See Theorem \ref{cst-wald})
$$L(1, E, \chi)=2^{-\mu(N, D)} \cdot \frac{8\pi^2 (\phi, \phi)_{\Gamma_0(N)}}{u^2 \sqrt{|Dc^2|}} \cdot \frac{|P_\chi(f)|^2}{\pair{f,f}},$$
Moreover, the Gross points of different conductors also form an ``Euler system''(See Section 2).
The following theorem can be viewed as the rank zero version of Theorem \ref{th:one}:
\begin{thm}\label{th:zero}
Let $M=q_1\cdots q_r$ be a positive square-free integer with prime factors $q_i$ all inert in $\BQ(\sqrt{-3})$ and $q_i\equiv 1\pmod 4$, then
\begin{enumerate}
  \item $\ord_{s=1}L(s,E^{(M)})=0=\rank E^{(M)}(\BQ)$;
  \item $\#\Sha\left(E^{(M)}/\BQ\right)$ is odd, and the $p$-part of the full BSD conjecture of $E^{(-\ell M)}$ holds for $p\not= 3$.
\end{enumerate}
\end{thm}

\s{\bf Acknowledgements.} The authors greatly thank Professor Ye Tian for suggesting this problem and his persistent encouragement.

\section{\bf The Modular Curve and Heegner points}

\subsection{The Modular Curve $X_K(N)$}\label{section-X_K(N)}
Let $K$ be an imaginary quadratic field with discriminant $D$. Let $N = N_0N_1^2$ be a positive integer
such that $p|N_1$ if and only if $p$ is inert in $K$. Let $c$ be another positive integer coprime to $N$.
Take an embedding $K\hookrightarrow M_2(\BQ)$ which is admissible in the sense that
$$K\cap M_2(\BZ)=K\cap R_0(N_0)=\CO_c.$$
Denote by $R=\CO_c+N_1R_0(N_0)$  an order of $M_2(\BQ)$. Then $R$ has discriminant $N$ with $R\cap K = \CO_c$.

Let $\Gamma_K(N) = R^\times \cap \SL_2(\BZ)$ and $X_K(N)$ be the modular curve over $\BQ$ with level $\Gamma_K(N)$.
It's well known that $X(N_0N_1)(\BC)$ parameterizes  $(E,(\BZ/N_0N_1)^2\simeq E[N_0N_1]))$
where $E$ is an elliptic curve over $\BC$. By \cite{KM}, it parameterizes $(E,(\BZ/N_0)^2\simeq E[N_0],(\BZ/N_1)^2\simeq E[N_1]))$.
Then $X_K(N)$ parameterizes $(E,C,\alpha:(\BZ/N_1)^2\simeq E[N_1]))$,
where $C$ is a cyclic subgroup of $E[N_0]$ of order $N_0$,
and $\alpha$ is an $H$-orbit of a basis of $E[N_1])$ where $H := (\CO_K/N_1\CO_K)^\times \subset \GL_2(\BZ/N_1\BZ)$.
Precisely, the class of $z\in\msh$ in $X_K(N)$ corresponds to the triple
$$\left(\BC/\BZ\cdot z+\BZ,\Big\langle\frac{1}{N_0}\Big\rangle,H\left({\frac{z}{N_1}\atop \frac{1}{N_1}}\right)\right).$$

\begin{lem}\label{lem:He}
If $m$ is a positive integer and $(m,cND_K)=1$, then for any invertible fractional ideals $\fa,\CN$ of $\CO_{cm}$,
satisfying $\CN^{-1}\fa/\fa\simeq \BZ/N_0\BZ$,  there exist a $\BZ$-basis $\{u,v\}$ of $\fa$ and $g \in\GL_2^+(\BQ)$,
such that $\CN^{-1}\fa=\BZ\dfrac{u}{N_0}+\BZ v$, $\dfrac{v}{u}=g^{-1}h_0$, and $K\cap gRg^{-1}= \CO_{cm}$.
\end{lem}
\begin{proof}
Consider the curve morphism $X_K(N)\to X_0(N_0)$ over $\BQ$, which is the forgetful functor in the moduli aspect $(E, C, [\alpha])\mapsto(E,C)$. $\fa,\CN$ defines a Heegner point $(\BC/\fa,\CN^{-1}\fa/\fa)$ on $X_0(N_0)$. Then exists a $\BZ$-basis $\{u,v\}$ of $\fa$ and $g=\mat(a,b,c,d)\in\GL_2^+(\BQ)$, $\CN^{-1}\fa=\BZ\pair{\dfrac{u}{N_0},v}$,  $\dfrac{v}{u}=g^{-1}h_0$, and $K\cap g\wh{R_0(N_0)}g^{-1}= \CO_{cm}$. Different choice of basis $(u,v)$ make $g$ differ an element in $\Gamma_0(N_0)=R_0(N_0)^\times$ and a scalar. So we have to prove there exists $g'\in \Gamma_0(N_0)$, such that $K\cap gg'\wh{R_0(N_0)}g^{'-1}g^{-1}= \CO_{cm}$.

In fact, we choose $g'\in \Gamma_0(N_0)$, such that $gg'\in \prod_{\ell|N_1}K_\ell^\times(1+N_1 M_2(\BZ_\ell))$, embedding $gg'$ into $\prod_{\ell|N_1}\GL_2(\BQ_\ell)$. Note that if $g^{-1}=\mat(a',b',c',d')$, and multiply a scalar if necessary, we may assume $g^{-1}\in M_2(\BZ)$, then $\BZ(a'h_0+b')+\BZ(c'h_0+d)$ and $\fa$ belong to a same class in $\Pic(\CO_{cm})$, hence
$$(\det g^{-1})m^{-1}=[\CO_{cm}:\fa]:=\dfrac{[\CO_{cm}:\CO_{cm}\cap\fa]}{[\fa:\CO_{cm}\cap\fa]};$$
Therefore there exists an integer $N_1'$ whose prime factors are prime factors of $N_1$, s.t. $\ell\nmid\det(N_1'g^{-1}),\forall \ell|N_1$. So $N_1'g^{-1}\in\GL_2(\BZ_\ell),\forall\ell|N_1$. By strong approximation, $\Gamma_0(N_0)\prod_{\ell|N_1}\CO_{K,\ell}^\times(1+N_1 M_2(\BZ_\ell))=\prod_{p|N_1}\GL_2(\BZ_\ell)$, then there exists $g'\in \Gamma_0(N_0), X\in\prod_{\ell|N_1}\CO_{K,\ell}^\times(1+N_1 M_2(\BZ_\ell))$, such that $g' X=N_1'g^{-1}$, thus $gg'=N_1'X\in\prod_{\ell|N_1}K_\ell^\times(1+N_1 M_2(\BZ_\ell))$. The proof is completed
\end{proof}

\begin{defn}
Let $\fa,\CN$ and $u,v,g$ be as in the above lemma.
A Heegner point on $X_K(N)$ of conductor $cm$ is a triple
$$\displaystyle{P=\left(\BC/\fa,\CN^{-1}\fa/\fa,H\left({\frac{v}{N_1}\atop\frac{u}{N_1}}\right)\right)}.$$
\end{defn}
\begin{remark}
  This point corresponds to the point $\dfrac{v}{u}\in\msh$.
\end{remark}

Let $h_0$ be the point $\msh^{K^\times}$. The order $\CO_c = \BZ + \BZ\varpi_c$ where $\varpi_c = \frac{cD+c\sqrt{D}}{2}$.
Denote by $\mat(x,y,z,w) \in \GL_2(\BQ)$ the image of $\varpi_c$ under the fixed embedding $K \hookrightarrow M_2(\BQ)$.
Since $K$ is a field, $z \not=0$.
\begin{lem}
$K=\BQ+\BQ h_0$ and $\CO_c=\BZ+\BZ yh_0^{-1}$.
\end{lem}
\begin{proof}
We have $x+w=Dc,xw-yz=\dfrac{c^2(D^2-D)}{4}$. As $h_0$ is fixed by $\mat(x,y,z,w)$,
\[\dfrac{xh_0+y}{zh_0+w}=h_0 \text{ and } zh_0^2+(w-x)h_0-y=0.\]

Hence
$$h_0=\dfrac{(x-w)+c\sqrt{D}}{2z}\in K\bs\BQ\text{ and }h_0^{-1}=\dfrac{2z}{(x-w)+c\sqrt{D}}=\dfrac{(x-w)-c\sqrt{D}}{2y},$$
so
$$yh_0^{-1}=-w+\dfrac{Dc+c\sqrt{D}}{2}.$$
\end{proof}

\begin{lem}
Let $\fa=\BZ+\BZ\cdot h_0^{-1}$ and $\CN^{-1}=\BZ+\BZ\cdot N_0^{-1}h_0^{-1}$.
Then $\End(\fa)=\{x\in K:x\fa\subset\fa\}=\CO_c$, and $\End(\CN^{-1})=\CO_c$.
\end{lem}
\begin{proof}
Let $(a+bh_0^{-1})\fa\subset\fa$. Then it is equivalent to
$$(a+bh_0^{-1})\in\fa\text{ and }(a+bh_0^{-1})h_0^{-1}\in\fa.$$
The first condition implies $a,b\in\BZ$, then the second one is equivalent to $bh_0^{-2}\in\fa$. But
$$bh_0^{-2}=y^{-1}b((w-x)h_0^{-1}+z)\in\fa.$$
The condition $R\cap K=\CO_c$ tells $(w-x,z,y)=1$, so the above condition implies $b\in y\BZ$, which is exactly $\End(\fa)=\CO_c$.
The assertion for $\CN^{-1}$ is the same, noticing that $N_0|z$.
\end{proof}

Clearly, $\fa$ and $\CN$ are invertible ideals of $\CO_c$, and $\CN^{-1}/\fa\simeq \BZ/N_0\BZ$. Summing up:
\begin{prop}
Let  $K\hookrightarrow M_2(\BQ)$ be an admissible embedding and $h_0 \in \CH^{K^\times}$.
Denote by $\fa=\BZ+\BZ\cdot h_0^{-1}$ and $\CN^{-1}=\BZ+\BZ\cdot N_0^{-1}h_0^{-1}$, then
$$\displaystyle{P=\left(\BC/\fa,\CN^{-1}/\fa,H\left({\frac{y}{N_1}\atop\frac{yh_0^{-1}}{N_1}}\right)\right)}$$
is a Heegner point on $X_K(N)$ of conductor $c$.
\end{prop}

\begin{example}\label{example-embed}
Now we construct an admissible embedding $K\hookrightarrow M_2(\BQ)$ as following. Since $\ell|N_0$ implies that $\ell$ is split in $K$, there exists an integral ideal $\fN_0$ of $\CO_K$ such that $\CO_K/\fN_0\simeq\BZ/N_0$, which implies $\BZ+\fN_0=\CO_K$. Then there exists $n\in\BZ$ and $m\in\fN_0$ such that
$$\frac{D+\sqrt{D}}{2}=n+m.$$
Take trace and norm, we get
$$D=2n+(m+\bar{m})\text{ and }\frac{D^2-D}{4}=n^2+n(m+\bar{m})+m\bar{m}$$
Since $m\bar{m}\in\fN_0\ov{\fN_0}=N_0\CO_K$ and it is an integer, so $m\bar{m}=N_0b$ for some $b\in\BZ$. Let $a=D-2n$, we see
$$D=a^2-4N_0b$$
It's easy to check that $(a,b,N_0)=1$. Given an integer $c$ such that $(c,N)=1$,  we let the embedding $i_c:K\longrightarrow B$ be given by
$$\xymatrix{
\dfrac{D+\sqrt{D}}{2}\ar@{|->}[r] & {\mat(\frac{D+a}{2},-c^{-1},N_0bc,\frac{D-a}{2})}
},\text{ or }\xymatrix{
\dfrac{Dc+\sqrt{Dc^2}}{2}\ar@{|->}[r] & {\mat(\frac{Dc+ac}{2},-1,N_0bc^2,\frac{Dc-ac}{2})}
}.$$
We can see this embedding is normal in the sense of \cite{Sh}, and $\displaystyle{h_0=\frac{a+\sqrt{D}}{2N_0bc}}$.
\end{example}

The modular curve $X_K(N)$ depends on the admissible embedding $K\hookrightarrow M_2(\BQ)$. However, we will prove that, all those modular curves given by admissible embeddings are isomorphic over $\BQ$. Let $i:K\hookrightarrow M_2(\BQ)$ be an admissible embedding, $H=i(\CO_K/N_1\CO_K)\subset\GL_2(\BZ/N_1\BZ)$, and $H_0$ be the upper-triangular matrices in $\GL_2(\BZ/N_0\BZ)$, then $H=\prod_{p|N_1}H_p$, where $H_p\subset\GL_2(\BZ/p^{\ord_pN_1}\BZ)$. Then
$$X_K(N)=X(N_0N_1)/(H\times H_0).$$

If $i'$ is another admissible embedding, and for any $p|N_1$, $H_p$ and $H_p'$ are conjugate in $\GL_2(\BZ/p^{\ord_pN_1}\BZ)$, then we obviously have
$$X(N_0N_1)/(H\times H_0)\simeq X(N_0N_1)/(H'\times H_0)$$

In fact, $H_p$ is the image of following kind of morphism:$\BZ_{p^2}^\times\to\GL_2(\BZ_p)\to\GL_2(\BZ_p/p^{\ord_p(N_1)}\BZ_p)$, then the following lemma implies that $H_p$ and $H_p'$ are conjugate in $\GL_2(\BZ/p^{\ord_pN_1}\BZ)$.

\begin{lem}
For any two embeddings $\varphi_i:\BZ_{p^2}\to M_2(\BZ_p), i=1,2$, there exists $g\in\GL_2(\BZ_p)$, such that $\varphi_1=g^{-1}\varphi_2 g$.
\end{lem}
\begin{remark}
If we change $\BZ_p$ to $\BQ_p$, and $\BZ_{p^2}$ to $\BQ_{p^2}$, this lemma is well-known.
\end{remark}
\begin{proof}
Consider $V=\BZ_p\oplus\BZ_p$, with a natural action of $M_2(\BZ_p)$. Via $\varphi_i$, we view $V$ as an$\BZ_{p^2}$-module, denoted by $V_i,i=1,2$. Since $\BZ_{p^2}$ is a discrete valuation ring and $V_i$ are torsion free, so $V_1,V_2$ are both free $\BZ_{p^2}$-module of rank 1, so there exists an isomorphism $g:V_1\to V_2$ of $\BZ_{p^2}$-module, this isomorphism corresponds to an element of $\GL_2(\BZ_p)$, also denoted by $g$. $g$ is an isomorphism of $\BZ_{p^2}$-modules means that
$$g\varphi_1(x)=\varphi_2(x)g,\forall x\in\BZ_{p^2}.$$
\end{proof}

\begin{lem}\label{lem-cusp}
   Let $\zeta_{N_1}$ be a primitive $N_1$-th root of unity. Then	
   the cusp $\infty$ of $X_K(N)$ is defined over $\BQ(\zeta_{N_1})$.
\end{lem}
\begin{proof}
In the adelic language, we have the following complex uniformization
\[X_K(N)(\BC) = \GL_2(\BQ)_+ \bs \CH \times \GL_2(\wh{\BZ})/ \wh{R}^\times \cup \{\mathrm{cusps}\}\]
where $\wh{R}  = R \otimes_\BZ \wh{\BZ}$ and the cusps are
\[\GL_2(\BQ)_+ \bs \BP^1(\BQ) \times \GL_2(\wh{\BZ}) / \wh{R}^\times.\]
The cusps are all defined over $\BQ^\ab$.
By \cite[pp.507]{Sch},
if we let $r:\wh{\BQ}^\times/\BQ^\times\to\Gal(\BQ^\ab/\BQ)$ be the Artin map,
then $r(x)\in \Gal(\BQ^\ab/\BQ)$ acts on the cusps by left multiplication the matrix $\mat(x,0,0,1)$.
Since $\wh{\BQ}^\times/\BQ^\times\simeq\wh{\BZ}^\times$, then if $x\in\wh{\BZ}^\times$ such that $r(x)\cdot[\infty,1]=[\infty,1]$,
there exists $\mat(\alpha,\beta,0,\gamma)\in\GL_2^+(\BQ)$, such that $\mat(\alpha,\beta,0,\gamma)\mat(x,0,0,1) \in \wh{R}^\times$,
which implies
$$\gamma\in\BZ_p^\times,\  \alpha x_p\in\BZ_p^\times,\  \beta\in N_1\BZ_p\text{ for all }p\text{, and }\alpha x_p\equiv\gamma\pmod{N_1}\text{ for all }p|N_1.$$
Then $\alpha = \gamma=\pm 1$, and $x_p\equiv  1\pmod {N_1}$. So the definition field of $[\infty,1]$ corresponds to $$\wh{\BQ}^\times/\BQ^\times\BZ^{\times(N_1)}\prod_{p|N_1}(1+N_1\BZ_p),$$
via class field theory, which is $\BQ(\zeta_{N_1})$.
\end{proof}

In the following, we fix the embedding $i_c:K\hookrightarrow M_2(\BQ)$.

\s{\bf{Atkin-Lehner operator.}}
Take $j=\mat(1,0,a,-1)$, then $kj=j\ov{k}$ for all $k\in K$ where $\ov{k}$ is the Galois conjugation of $k$.
For each $p\mid N_0$, let
\[w_p=\mat(0,1,p^{\ord_pN},0) \in \GL_2(\BQ_p)\]
be the local Atkin-Lehner operator. Define
\[w=j^{(N_0)}\cdot\prod_{p\mid N_0}w_p\in \GL_2(\wh{\BQ})\cap M_2(\wh{\BZ}).\]
Since $w$ normalizes $\wh{R}^\times$, it acts on $X_{K}(N)$.
%For all $p$, we let $w_p$ be the $p$-component of $w$.
%This makes no confusion when $p|N_0$, since the two $w_p$'s we have just mentioned coincide.
%Then the Heegner points of conductor $c$ are stable under the action of $w_p$, for all $p|N$.

For each $p|N_0$, write $N_0=p^km$ with $(p,m)=1$.
Similar to the proof of lemma (\ref{lem:He}), we can choose $u,v\in\BZ$ such that $p^ku+mv=1$,
let $g=\mat(p^k,1,-N_0 v,p^ku)\in\GL_2^+(\BQ)\cap M_2(\BZ)$, such that $g^{-1}w_p\in U$, so
$$w_p P=[h_0,w_p]=[g^{-1}h_0,1]=\left(\BC/\BZ\cdot g^{-1}h_0+\BZ,\Big\langle\frac{1}{N_0}\Big\rangle,H\left({\frac{g^{-1}h_0}{N_1}\atop \frac{1}{N_1}}\right)\right)$$
Modify it, we get $w_p P=\left(\BC/\BZ h_0^{-1}+\BZ\cdot p^k,\Big\langle v+\frac{h_0^{-1}}{m}\Big\rangle,H\left({\frac{p^k}{N_1}\atop \frac{p^kh_0^{-1}}{N_1}}\right)\right)$. However, $\fa=\BZ h_0^{-1}+\BZ\cdot p^k=\fp^k$ is an invertible ideal of $\CO$ dividing $\CN$. Suppose $\CN=\fa\fm$, let $\CN^\prime=\ov{\fa}\fm$, then $\CN^{\prime-1}\fa/\fa=\Big\langle v+\frac{h_0^{-1}}{m}\Big\rangle$.

In another words, consider the quotient map $\xi:X_K(N)\to X_0(N_0)$ induced by $\Gamma_K(N)\subset\Gamma_0(N_0)$, which is defined over $\BQ$, then the above argument says that $\xi\circ w_p=w_p\circ\xi$, where the action of $w_p$ on $X_0(N_0)$ is defined by \cite[pp.90]{GR}.

To study the action of $w$ on $X_K(N)$, we can prove the following lemma:
\begin{lem}\label{Bir}
	There exists $t_0\in\wh{K}^\times$ and $u\in \wh{R}^\times$ such that $w=t_0ju$.
\end{lem}
\begin{proof}
For $p|N_0$, let $k=\ord_p N_0\geq 1$, $K_p^\times=(\BQ_p+\BQ_p(\sqrt{D}))^\times$. Let $x,y\in\BQ_p$, then
$$(x+y\sqrt{D})j_p^{-1}w_p=\begin{pmatrix}-2p^k y & x+ay \\ -p^k(x+ay) & a(x+ay)-2N_0by\end{pmatrix}$$
So we choose $\ord_py=-k-\ord_p2,x+ay\in\BZ_p$ and such that $a(x+ay)\in\BZ_p^\times$. Then let $t_{0,p}=x+y\sqrt{D}$.

For $p\nmid N_0$, let $t_{0,p}=1$. Then such choice of $t_0$ works.
\end{proof}
\begin{remark}
By Shimura reciprocity law, if we use $[x]\mapsto\ov{[x]}$ to denote the complex conjugation on $X_K(N)(\BC)$, then
$$\ov{[h_0,g]}=[h_0,jg],\forall g\in\GL_2(\BA_f).$$
Lemma \ref{Bir} in fact tells that the action of $w$ is the composition of a Galois action and the complex conjugation.
\end{remark}

\s{\bf{Hecke correspondence.}} Let $\ell\nmid N$ be a prime, the Hecke correspondence on $X_U$ is defined by
$$T_\ell\left(E,C,H\left({x_1\atop x_2}\right)\right)=\sum_i\left(E/C_i,(C+C_i)/C_i,H\left({x_1\mod C_i}\atop {x_2\mod C_i}\right)\right)$$
where the sum is taken over all cyclic subgroups $C_i$ of $E$ of order $\ell$, $\alpha_i$ is given by $(\BZ/N_1\BZ)^2\stackrel{\alpha}{\longrightarrow} E[N_1]\simeq (E/C_i)[N_1]$. $E[N_1]\simeq (E/C_i)[N_1]$ is because $(\ell,N)=1$. This is just by definition.

\subsection{Gross-Zagier Formula}
Let $E$ be an elliptic curve of conductor $N$, $K$ be an imaginary quadratic field of discriminant $D$ and
$\chi$ be a ring class character over $K$ of conductor $c$. Assume
\begin{center}\textit{$E,K,\chi$} satisfy the condition $(*)$\end{center}
Then we can write $N=N_0N_1^2$, where $p|N$ is inert in $K$ if and only if $p|N_1$.

Embed $K$ in $M_2(\BQ)$ by $i_c$. There is a modular parametrization $f:X_K(N) \ra E$ mapping $[\infty]$ to the identity of
$E$. If $f_1,f_2$ are two such morphisms, then there exist integers $n_1,n_2$ such that $n_1f_1 = n_2f_2$. Let
$h_0$ be the point in $\CH$ fixed by $K^\times$, then $\CO_c=\BZ +\BZ h_0^{-1},\ \CN = \BZ + \BZ N_0^{-1}h_0^{-1}$ is an invertible ideal of $\CO_c$ such that
$\CN/\CO_c \cong \BZ/N_0\BZ$. Consider the following Heegner point on $X_K(N)$ of conductor $c$
\[P = \left(\BC/\CO_c, \CN/\CO_c, H
	\begin{pmatrix}
		\frac{1}{N_1} \\
		\frac{h_0^{-1}}{N_1}
\end{pmatrix}\right) \in X_K(N)(H_c).\]
Form the cycle:
$$P_\chi(f)=\sum_{\sigma \in \Gal(H_c/K)}f(P^{\sigma})\chi(\sigma) \in E(H_c)\otimes_\BZ \BC.$$

\begin{thm}[Explicit Gross-Zagier formula \cite{CST}]\label{GZ}
We have the following equation
$$L'(1,E,\chi)=2^{-\mu(N,D)}\cdot\frac{8\pi^2(\phi,\phi)_{\Gamma_0(N)}}{u^2c\sqrt{|D|}}\frac{\wh{h}_K(P_\chi(f))}{\deg f} \eqno(\rm{GZ})$$
here $\phi$ is the normalized newform associated to $E$, $\mu(N,D)$ is the number of prime factors of $(N,D)$,
$u = [\CO_K^\times:\BZ^\times]$, $\wh{h}_K$ is the Neron-Tate height pairing over $K$ and the Petersson inner product
\[(\phi,\phi)_{\Gamma_0(N)} = \int_{\Gamma_0(N) \bs \CH} |\phi(x+iy)|^2dxdy.\]
\end{thm}

\subsection{Euler System}\label{euler}
Let $S$ be a finite set of primes containing the prime factors of $6cND_K$, $\BN_S$ denote the set of integers any of whose prime divisors is not in $S$.
For any $\ell, m\in \BN_S$ with $\ell$ a prime and $\ell\nmid m$, let $P_m=(\BC/\fa_m,\CN_m^{-1}\fa_m/\fa_m,\alpha_m)$ be a Heegner point of conductor $cm$. Let $P_{m\ell}=(\BC/\fa_{m\ell},\CN^{-1}_{m\ell}\fa_{m\ell}/\fa_{m\ell},\alpha_{m\ell})$, such that $\CN_{m\ell}=\CN_m\cap\CO_{m\ell},\fa_{m\ell}=\fa\cap\CO_{m\ell}$, $\alpha_{m\ell}$ is the composition
$$(\BZ/N_1\BZ)^2\stackrel{\alpha_m}{\longrightarrow}N_1^{-1}\fa_m/\fa_{m}\stackrel{\sim}{\longrightarrow}N_1^{-1}\fa_{m\ell}/\fa_{m\ell}.$$
\begin{thm}\label{Nr:He}
Then we have that $[H_{m\ell}: H_m]=(\ell+1)/u_m$ if $\ell$ is inert in $K$ and
$(\ell-1)/u_m$ if $\ell$ is split and
$$u_m \sum\limits_{\sigma\in \Gal(H_{m\ell}/H_m)} P^\sigma_{m\ell} =\begin{cases}
\RT_\ell P_m, \qquad &\text{if $\ell$ is inert in $K$},\\
(\RT_\ell-\sum_{w|\ell} \Frob_w)P_m, \qquad &\text{if $\ell$ is
split in $K$},
\end{cases}$$
where $\RT_\ell$ is the Hecke correspondence, $\Frob_w$ is the Frobenius at $w|\ell$ in $\Gal(H_m/K)$,
and $u_m=1$ if $m\neq 1$ and $u_1=[\CO_K^\times: \BZ^\times]$.
\end{thm}
This theorem is proved in general by \cite[Proposition 4.8]{Ne} or \cite[Theorem 3.1.1]{Tian3}.

\subsection{Waldspurger Formula and Gross Points}\label{section-walds}

Let $\phi=\sum_{n=1}^\infty a_n q^n$ be a newform of weight $2$, level $\Gamma_0(N)$, normalized such that $a_1=1$.
Let $K$ be an imaginary quadratic field of discriminant $D$ and
$\chi$ a  ring class character over $K$ of conductor $c$.
Let $L(s,\phi,\chi)$ be the Rankin-Selberg convolution of $\phi$ and $\chi$.

Assume that $(c,N) = 1$. Denote by $\BS$ the set of primes $p|N$ satisfying one of the following conditions:
\begin{itemize}
	\item $p$ is inert in $K$ with $\ord_p(N)$ odd;
	\item $p|D$, $\ord_p(N) = 1$ and $\chi([\fp]) = a_p$ where $\fp$ is the prime of $\CO_K$ above $p$
		and $[\fp]$ is its class in $\Pic(\CO_c)$;
	\item $p|D$, $\ord_p(N) \geq 2$ and the local root number of $L(s,\phi,\chi)$ at $p$ equals $-\eta_p(-1)$
	     where $\eta_p$ is the quadratic character for $K_p/\BQ_p$.
\end{itemize}
Then the sign of $L(s,\phi,\chi)$ is $+1$. Let $B$ be
the definite quaternion algebra defined over $\BQ$ ramified exactly at primes in $\BS\cup\{\infty\}$. Fix an embedding
from $K$ in $B$. Let $R$ be an order in $B$ with discriminant $N$ and $R \cap K = \CO_c$. Denote by
$\hat{R} = R \otimes_\BZ \hat{\BZ}$ and $U = \hat{R}^\times$ which is an open compact subgroup of $\hat{B}^\times$.
Consider the Shimura set $X_U = B^\times \bs \hat{B}^\times /U$ which is a finite set. A point in
$X_U$ represented by $x \in \hat{B}^\times$ is denoted by $[x]$.
Note that for $p|(D,N)$, $K_p^\times$ normalizes $U$ and then $K_p^\times$ acts on $X_U$ by
right multiplication. Let
\[\BC[X_U]^0 = \left\{ f \in \BC[X_U] \Bigg| \sum_{x \in X_U} f(x) = 0 \right\}.\]
For each $p \not| N$, there are Hecke correspondences $T_p$ and $S_p$. In this case, $B_p$ is split while
$U_p$ is maximal. Then the quotient $B_p^\times/U_p$ can be identified with $\BZ_p$-lattices in $\BQ_p^2$.
Then for any $[x] \in X_U$,
\[S_p[x] := [x^{(p)}s_p], \quad T_p[x] := \sum_{h_p} [x^{(p)}h_p]\]
where if $x_p$ corresponds to a lattice $\Lambda$, then $s_p$ is the lattice $p\Lambda$ and
the set $\{ h_p \}$ is the set of  sublattices $\Lambda'$ of $\Lambda$ with $[\Lambda:\Lambda'] = p$.
There is then a line $V(\phi,\chi)$ of $\BC[X_U]^0$ characteristized as following
\begin{itemize}
	\item for any $p \not| N$, $T_p$ acts on $V(\phi,\chi)$ by $a_p$ and $S_p$ acts trivially;
	\item for any $p|(D,N)$ with $\ord_p(N) \geq 2$, $K_p^\times$ acts on $V(\phi,\chi)$ by $\chi_p$.
\end{itemize}
Let $f$ be a nonzero vector in $V(\phi,\chi)$ and consider the period
\[P_\chi(f) = \sum_{\sigma \in \Gal(H_c/K)} f(\sigma)\chi(\sigma)\]
where the embedding of $K$ to $B$ induces a map
\[\Gal(H_c/K) = K^\times \bs \hat{K}^\times /\wh{\CO}_c^\times \lra X_U.\]

\begin{thm}[Explicit Waldspurger formula \cite{CST}]
\label{cst-wald}
We have the following equation
$$L(1,\phi,\chi)=2^{-\mu(N,D)}\cdot\frac{8\pi^2(\phi,\phi)_{\Gamma_0(N)}}{u^2c\sqrt{|D|}}
\frac{|P_\chi(f)|^2}{\langle f,f \rangle} \eqno(\rm{Wald})$$
where the pairing
\[\langle f,f \rangle = \sum_{[x] \in X_U} |f(x)|^2 w([x])^{-1}\]
and $w([x])$ is the order of the finite group $(B^\times \cap xUx^{-1})/\{ \pm 1\}$.
\end{thm}

There is an analogue to Heegner points, the so called Gross points.
Let $S$ be a set of finite places of $\BQ$ containing all places
dividing $6cND$. Let $\BN_S$ denote the set of integers  whose prime divisors
are not in $S$.
\begin{defn}
Let $m \in \BN_S$. A point $x_m \in K^\times \bs \hat{B}^\times/ U$ is called a \textbf{Gross Point}
of conductor $cm$, if $x_mUx_m^{-1}\cap \hat{K}^\times=\wh{\CO}_{cm}^\times$.
\end{defn}
Each element in $K^\times\bs\wh{K}^\times/\wh{\CO}_{cm}^\times$ acts on $x_m$ by left multipication.
This induces an action of $\Gal(H_{cm}/K)$ on $x_m$, also called the Galois action.

For each prime $\ell\in\BN_S$,
fix an isomorphism $\beta_\ell:B_\ell \stackrel{\sim}{\ra} M_2(\BQ_\ell)$,
such that $\beta_\ell(U_\ell)=\GL_2(\BZ_{\ell})$, and, under this isomorphism, we have
\begin{itemize}
  \item $\beta_\ell(K_\ell)=\left\{\mat(a,0,0,b):a,b\in \BQ_\ell\right\}$, if $\ell$ is split in $K$;
  \item $\beta_\ell(K_\ell)=\left\{\mat(a,b\delta,b,a):a,b\in \BQ_\ell\right\}$, where
	  $\delta\in\BZ_{p}^\times\bs\BZ_{p}^{\times 2}$, if $\ell$ is inert in $K$.
\end{itemize}

%For $\ell\nmid N$, we define the Hecke correspondences $T_\ell, U_\ell$ on $B^\times\bs\wh{B}^\times/U$ as following:
%$$\begin{array}{l}
%  T_\ell([b])=\sum\limits_{i\in\CO_F/\ell\CO_F}\left[b\mat(\ell,i,0,1)\right]+ \left[b\mat(1,0,0,\ell)\right] \\
%  U_\ell([b])=\sum\limits_{i\in\CO_F/\ell\CO_F}\left[b\mat(\ell,i,0,1)\right]
%\end{array}$$
For $m\in\BN_S$, define $x_m\in\wh{B}^\times$ by
 $$(x_m)_{\ell}=\left\{
     \begin{array}{ll}
       \beta_\ell^{-1}\mat(\ell^{\ord_\ell m},0,0,1) & \ell|m \\
       1 & \ell\nmid m
     \end{array}
   \right.
 $$
 Then the image of $x_m$ in $K^\times \bs \wh{B}^\times /U$, still denoted by $x_m$, is a Gross point of
 conductor $cm$.

\begin{thm}\label{Nr:Gr}
For any $\ell, m\in \BN_S$ with $\ell$ a prime and $\ell\nmid m$,  we have that
$$u_m \sum_{\sigma \in \Gal(H_{cm\ell}/H_{cm})} [\sigma.x_{m\ell}] =\begin{cases}
T_\ell [x_m], \qquad &\text{if $\ell$ is inert in $K$},\\
(T_\ell-\sum_{w|\ell} \Frob_w)[x_m], \qquad &\text{if $\ell$ is
split in $K$},
\end{cases}$$
where the equality holds as divisors on $X_U$, with $\Frob_w$ and $u_m$ the same as Theorem \ref{Nr:He}.
\end{thm}
The proof is the same as the norm relation of Heegner points on Shimura curves. One can refer to \cite[Proposition 4.8]{Ne} or \cite[Theorem 3.1.1]{Tian3}.

\section{\bf Quadratic Twists of $X_0(36)$}
%\section{Explicit formula}

%\subsection{Preliminary on $X_0(36)$}\label{Pre}
The modular curve $X_0(36)$ has genus one and its cusp $[\infty]$ is
rational over $\BQ$ so that $E=(X_0(36), [\infty])$ is an elliptic
curve defined over $\BQ$. The elliptic curve $E$ has CM by
$\BQ(\sqrt{-3})$ and has minimal Weierstrass equation
$$\qquad y^2=x^3+1.$$
Note that its Tamagawa numbers are $c_2=3, c_3=2$ and $E(\BQ)\cong
\BZ/6\BZ$ is generated by the cusp $[0]=(2, 3)$, we use $T$ to denote the non-trivial 2-torsion point in the following. Denote by
$L^\alg(E,s)$ the algebraic part of $L(E,s)$. Then
$L^\alg(E,1)=1/6$.

For a non-zero integer $m$, let $E^{(m)}: y^2 = x^3 + m^3$
the quadratic twist of $E$ by
the field $\BQ(\sqrt{m})$. Then $E^{(m)}$ and $E^{(-3m)}$ are 3-isogenous
to each other.

\begin{lem}Let $D\in \BZ$ be a fundamental discriminant of a
quadratic field. Then the sign
for the functional equation of $E^{(D)}$, denoted by $\epsilon(E^{(D)})$, is
\[(-1)^{\#\left\{ p|D, p = 2,3,\infty \right\}}\]
where $\infty|D$ means that $D < 0$.
\end{lem}
\begin{proof}
	For each $D$, denote by $K = \BQ(\sqrt{D})$, then
	\[L(s,E_K) = L(s,E)L(s,E^{(D)})\]
	where $L(s,E_K)$ is the base change $L$-function and it suffices to
	determine the sign of $L(s,E_K)$. Note that the local
	components of the cuspidal automorphic representation for $E$ at
	places $2$ and $3$ are supercuspidal with conductor $2$, then by
	\cite[Proposition 3.5]{T}, the local root number for the base
	change $L$-function at places $2$ (resp. $3$) is negative if and only if
	$2 | D$ (resp. $3 | D$). Meanwhile, the local root number at $\infty$ is positive
	if and only if $D$ is positive and for any place not dividing $6\infty$ it is
	positive. Summing up, the result holds.
\end{proof}

\subsection{The Waldspurger Formula}
Let $B$ be the definite quaternion algebra over $\BQ$ ramified at $3,\infty$, then we know that
$$B=\BQ+\BQ i+\BQ j+\BQ k,\quad i^2=-1,j^2=-3,k=ij=-ji. $$
Let $\CO_B=\BZ[1,i,(i+j)/2,(1+k)/2]$ of $B$.
The unit group $\CO_B^\times$ of $\CO_B$ equals to
\[\left\{ \pm 1, \pm i, \pm (i+j)/2, \pm (i-j)/2, \pm (1+k)/2,
\pm (1-k)/2\right\}.\]
Let $K = \BQ(\sqrt{-3})$ and $\eta: \wh{\BQ}^\times/\BQ^\times\to\{\pm 1\}$ is the quadratic character associated to $K$.
Embed $K\hookrightarrow B$ by sending $\sqrt{-3}$ to $k$,
which induces an embedding $\wh{K}^\times\hookrightarrow\wh{B}^\times$.

Let $\pi = \otimes_v \pi_v$ be the automorphic representation of $B_\BA^\times$ corresponding to $E$ via the modularity of $E$
and the Jacquet-Langlands correspondence.
Let $\CR = \prod_p \CR_p$ be an order of $\wh{B}^\times$ defined as following.
If $p = 2$, then $\CR_2 = \CO_{K,2} + 2\CO_{B,2}$. If $p = 3$, then $\CR_3 = \CO_{K,3} + \lambda\CO_{B,3}$ where
$\lambda \in B^\times$ is a uniformizer of $B_3$;
for example, we may choose $\lambda=k$, which is also a uniformizer of $K_3$. For $p \not| 6$, $\CR_p = \CO_{B,p}$.
Denote by $U = \CR^\times$. Then $U$ is an open compact subgroup of $\wh{B}^\times$.

The local components of $\pi$ have the following properties:
\begin{itemize}
  \item $\pi_\infty$ is trivial;
  \item $\pi_p$ is unramified if $p\neq 2,3,\infty$, i.e. $\pi^{\CO_{B,p}^\times}$ is one dimensional;
  \item $\pi_2^{\CO_{K,2}^\times}$ is one dimensional and $\pi_3^{\CO_{K,3}^\times}$ is two dimensional.
\end{itemize}
The first two properties are standard, while the last property comes from \cite[proposition 3.8]{CST}.
Then $\pi^U$ is a representation of $B_3^\times$ with dimension $2$.
As $K_3^\times$-modules, $\pi^U = \BC \chi_+ \oplus \BC \chi_-$
where $\chi_+$ is the trivial character of $K_3^\times$ and $\chi_-$ is the nontrivial quadratic unramified character on $K_3^\times$.

This representation $\pi^U$ is naturally realized as a subspace of the space of the
infinitely differentiable complex-valued functions $C^\infty(B^\times \bs \wh{B}^\times/\wh{\BQ}^\times)$.
The space $\pi^U$ is contained in the space $C^\infty(B^\times \bs \wh{B}^\times/ \wh{\BQ}^\times U)$
and is perpendicular to the spectrum consisting of characters (the residue spectrum).
In fact, we have the following more detailed proposition:

\begin{prop}\label{basis}
\begin{enumerate}
  \item $\pi^U$ has an orthonormal basis $f_+$, $f_-$  under the Petersson inner product defined by
\[\parallel f\parallel^2 = \int_{B^\times\bs \wh{B}^\times/\wh{\BQ}^\times } |f(g)|^2dg\]
with the Tamagawa measure $\Vol(B^\times\bs \wh{B}^\times/\wh{\BQ}^\times) = 2$.
  \item Moreover, $f_+$ (resp. $f_-$) is the function on $B^\times \bs \wh{B}^\times/\wh{\BQ}^\times U$,
	  supported on those $g \in \wh{B}^\times$ with $\chi_0(g) = +1$ (resp. $=-1$),
	  valued in $0, \pm 1$ with total mass zero, where $\chi_0$ is the composition of
	  the following morphisms:
	  \[B^\times\bs\wh{B}^\times/\wh{\BQ}^\times U \stackrel{\det}{\lra}   \wh{\BQ}^\times \stackrel{\eta}{\lra}   \pm 1.\]
  \item For any $t \in K_3^\times$, $\pi(t)f_+=\chi_+(t)f_+$ and $\pi(t)f_-=\chi_-(t)f_-$.
\end{enumerate}
\end{prop}
%\begin{remark}
%$\chi_0$ is a function on $B^\times \bs \wh{B}^\times/\wh{\BQ}^\times U$.
%\end{remark}
%\begin{remark}
%We know that $B^\times \bs \wh{B}^\times/\wh{\BQ}^\times U$ is a finite set. Let $\{g_i\}$ be a set of representatives, and $w_i=\#(B^\times\cap g_iUg_i^{-1})/2$, then the Peterson norm of $f\in\pi^U$ is
%$$\parallel f\parallel^2=\sum_i|f(g_i)|w_i^{-1}.$$
%\end{remark}

Since the class number of $B$ with respect to $\CO_B$ is $1$ by \cite[pp. 152]{V}, one has
\[\wh{B}^\times = B^\times\wh{\CO}_B^\times = B^\times B_3^\times \wh{\CO}_B^{\times(3)}.\]
Therefore,
\[
   B^\times \bs \wh{B}^\times/ \wh{\BQ}^\times U =
   B^\times \bs B^\times B_3^\times\wh{\CO}_B^{\times(3)} /U_2U_3\wh{\CO}_B^{\times(6)}
    = H \bs B_3^\times\CO_{B,2}^\times /U_2U_3,
\]
where $H=B^\times \cap B_3^\times
\wh{\CO}_B^{\times(3)}=\CO_B^\times \lambda^\BZ\subset B_3^\times\CO_{B,2}^\times$
and the last inclusion is given by the diagonal embedding.

\begin{lem}
The double coset $H\bs \CO_{B,2}^\times/U_2$ is trivial  and
$H\cap U_2 \bs B_3^\times/U_3 = \CO_{B,3}^\times/U_3$.
\end{lem}
\begin{proof}
  The proof is elementary. Firstly, we prove that $H \bs \CO_{B,2}^\times/U_2$ is trivial.
  Recall that  $U_2 = \CO_{K,2}^\times(1+2M_2(\BZ_2))$.
  As $\GL_2(\BZ_2)/(1+2M_2(\BZ_2)) = \GL_2(\BF_2)$,  the claim follows from
  that for any $g \in \GL_2(\BF_2)$, one may find $h \in H$ and $u \in \CO_{K,2}^\times$
  such that $g  \equiv hu \pmod{2\BZ_2}$.  For the second claim, note that
  \[H \cap U_2 = \langle k,-1,\frac{1+k}{2} \rangle, \]
  For any $x \in B_3^\times$, $x^{-1}(1+ \lambda)x = 1 + x^{-1}\lambda x \in U_3$
  where $\lambda$ is any uniformizer of $B_3^\times$.
  In particular, the action of $H\cap U_2$ on $B_3^\times/U_3$ is equal to the action of the group generated
  by some uniformizer. Hence
  $H \cap U_2 \bs B_3^\times/U_3 = H\cap U_2 \bs (B_3^\times/U_3)  = \CO_{B,3}^\times/U_3$.
\end{proof}

If we denote $\BZ_9$ the integer ring for the unramified quadratic extension field of $\BQ_3$, then
\[
  \CO_{B,3}^\times = \BZ_9^\times(1+\lambda\BZ_9); U_3 = \CO_{K,3}^\times(1+\lambda\CO_{B,3}) =
\mu_{2}(1+3\BZ_9)(1+\lambda\BZ_9)\]
where $\mu_2 = \{\pm 1\}$.
Hence
\[\begin{aligned}
    H\bs B_3^\times\CO_{B,2}^\times/U_2U_3 &\stackrel{\sim}{\lla} H\cap U_2 \bs B_3^\times /U_3 \\
    &\stackrel{\sim}{\lla} \CO_{B,3}^\times/U_3 \\
    &\stackrel{\sim}{\lla} \BZ_9^\times/\mu_{2}(1+3\BZ_9) \cong \BZ/4\BZ,
  \end{aligned}\]
and we can identify $C^\infty(B^\times \bs \wh{B}^\times/\wh{\BQ}^\times U)$ with $\BC[\BZ/4\BZ]$.

%Take $g_1,\cdots,g_4\in \wh{B}^\times$ which are trivial outside $2$ while $g_{k,2}$ is, for $k=1,2,3,4$ respectively, $1,i,1+i,1-i$,. Let $b_k\in B^\times$ be, for $k=1,2,3,4$ respectively, $1,i,1+i,1-i$.
%\begin{lem}
%$b_k^{-1}g_k Ug_k^{-1}b_k=U$, for $k=1,2,3,4$.
%\end{lem}
%\begin{proof}
%At the place $p=2$, $(b_k^{-1}g_k)_p=1$, so $(b_k^{-1}g_k)_2$ normalizes $U_2$. At the place $p=3$, Note that $iki^{-1}=-k$, $(1+i)k(1+i)^{-1}=j=ki$ and $(1-i)k(1-i)^{-1}=-j=k(-i)$, together with $U_3=\CO_{K,3}^\times+k\CO_{B,3}=\BZ_3^{\times}(1+k\CO_{K,3})+k\CO_{B,3}$, we know $(b_k^{-1}g_k)_3$ normalizes $U_3$. At the place $p\nmid 6$, $(b_k^{-1}g_k)_p\in\CO_{B,3}^\times$, while $U_p=\CO_{B,3}^\times$.
%\end{proof}
%A straight corollary of this lemma is that $w_k=\#(B^\times\cap g_k Ug_k^{-1})/2=\#(B^\times\cap U)/2$, for $k=1,2,3,4$, then one can easily prove that $w_k=3$, for $k=1,\cdots,4$. Therefore, the Peterson norm of $f\in\pi^U$ is
%$$\parallel f\parallel^2=\frac{1}{3}\sum_i|f(g_i)|.$$

The image of $B^\times \bs \hat{B}^\times/ U$ under the norm map is $\BQ_+^\times \bs \hat{\BQ}^\times/\Nr U$.
If $p \not = 3$, $\Nr U_p = \BZ_p^\times$ while if  $\Nr U_3 = 1+3\BZ_3$. Therefore,
by the approximation theorem,
\[\BQ_+^\times \bs \hat{\BQ}^\times/\Nr U = \BZ_3^\times/\Nr U_3 = \BZ_3^\times/1+3\BZ_3.\]
In paticular, the cardinality of $\BQ_+^\times\bs \hat{\BQ}^\times/\Nr U$ is $2$.
Forms in $C^\infty(B^\times \bs \wh{B}^\times/\wh{\BQ}^\times U)$
of the form $\mu\circ \Nr$ for some Hecke character $\mu$
correspond to characters on $\BC[\BZ/4\BZ]$ of order dividing $2$. Sum up, we obtain
\begin{lem}
  There is a natural bijection
  \[\BZ_9^\times/\mu_2(1+3\BZ_9) \stackrel{\sim}{\lra} B^\times \bs \wh{B}^\times/\wh{\BQ}^\times U\]
  which is induced by the embedding $\BZ_9^\times \ra B_3^\times \ra \wh{B}^\times$,
  and the left hand side of the above
  bijection is isomorphic to the cyclic group of order $4$. Via this bijection,
  the space $\pi^U$ is spanned by characters on the cyclic group with order not dividing $2$.
\end{lem}

Since $\CO_{K,3}^\times \subset U_3$, $f$ is $\chi_\pm$-eigen if and only if $\pi_3(\varpi_3)f = \pm f$,
if and only if $f(\zeta^a\varpi_3) = \pm f(\zeta^a)$ for $a = 0,\cdots,3$
where $\zeta$ is a primitive $8$th root of unity in $\BZ_9^\times$.
Moreover, we may assume $\zeta \equiv 1+i \pmod{1+3\BZ_9}$.

To compute $f(\zeta^a\varpi_3)$, since $k \in H\cap U_2$ and $f \in \pi^U$, we have
\[f(\zeta^a\varpi_3) = f(k^{-1}\zeta^a\varpi_3) = f(k_3^{-1}\zeta^a\varpi_3).\]
where $k_3$ denote the $3$-component of $k$.

Take $\varpi_3 = \sqrt{-3} \in K_3^\times$. Then
\[f(k_3^{-1}\zeta^a\varpi_3) = f(k_3^{-1}\zeta^a k_3) = f(\zeta^{3a}), \quad a \in \BZ/4\BZ\]
because the conjugate action of $k_3$ on $\BZ_9$ is the Galois conjugation. Thus
\[\pi(\varpi_3)f(\zeta) = f(\zeta^3), \quad \pi(\varpi_3)f(\zeta^a) = f(\zeta^a), \quad \text{if } 2a = 0.\]
Thus, one may take $f_+$ and $f_-$ by
$$\begin{array}{l}
  f_+(1) = 1, \quad f_+(\zeta^2) = -1, \quad f_+(\zeta) = f_+(\zeta^3) = 0 \text{ and }\\
  f_-(\zeta) = 1, \quad f_-(\zeta^3) = -1, \quad f_-(1) = f_-(\zeta^2) = 0.
\end{array}
$$

Finally, $\chi_0$ is the non-trivial element in the residue spectrum of $C^\infty(B^\times \bs \wh{B}^\times/\wh{\BQ}^\times U)$
and $\chi_0(\zeta^a) = (-1)^a$ for $a = 0,\cdots,3$.
Thus, up to $\pm 1$,  $f_+$ (resp. $f_-$) is the
function on $B^\times \bs \wh{B}^\times/\wh{\BQ}^\times U$,
supported on those $g \in \wh{B}^\times$ with $\chi_0(g) = +1$ (resp. $=-1$),
valued in $0, \pm 1$ with total mass zero. It is clear that
$f_+$ and $f_-$ is an orthonormal basis of $\pi^U$. We have
completed the proof of Proposition \ref{basis}.
\vspace*{5mm}

Now let $M = q_1\ldots q_r$ with $q_i \equiv 5 \pmod{12}$.
For any $q|M$, taking an isomorphism
$\iota_{q}: B_{q} \stackrel{\sim}{\ra} M_2(\BQ_{q})$ by
$i \mapsto\mat(1,0,0,-1)$ and $k \mapsto\mat(0,3,1,0)$.
In particular,$\iota_{q}(\CO_{B,q}) = M_2(\BZ_{q})$.
Denote by $x_{q} \in B_{q}^\times$ with
$\iota_{q}(x_{q}) =
\begin{pmatrix}
  q & \\
    & 1
  \end{pmatrix}$. Then $x_{q} \CO_{B,q} x_{q}^{-1} \cap K_{q} = \CO_{M,q}$.
Take $x_M = \prod_i x_{q_i} \in \wh{B}^\times$. Denote by
\[f_M =
  \begin{cases}
     f_+(\cdot x_M), \quad \text{if $r$ is even}; \\
     f_-(\cdot x_M), \quad \text{if $r$ is odd}.
 \end{cases}\]
Let $\chi_M$ be the quadratic Hecke character of $K$ associated to $K(\sqrt{M})/K$,
then $\chi_M(\varpi_3)=(-1)^r$. Then $f_M$ is in the line $V(\pi,\chi_M)$ defined in
Section \ref{section-walds}.
In particular, it satisfies that
\begin{enumerate}
  \item $\forall p\nmid 6M, T_pf_M=a_pf_M$;
  \item $f_M$ is integrable-valued with minimal norm;
  \item $\pi(\varpi_3)f_M=\chi_M(\varpi_3)f_M$.
\end{enumerate}

Let $H_M$ be the ring class field of $K$ of conductor $M$, i.e. the abelian extension of $K$
with Galois group $\Gal(H_M/K)\simeq\Pic(\CO_M)=\wh{K}^\times/K^\times\wh{\CO}^\times_M$.
The embedding $K \hookrightarrow B$ induces a map
\[K^\times \bs \wh{K}^\times / \wh{\CO}^\times_M \lra B^\times \bs \wh{B}^\times / U.\]
Consider
$$P_{\chi_M}(f_M)=\sum_{t\in\Pic(\CO_M)}f_M(t)\chi_M(t).$$
Denote by
\[L^\alg(s,E) = L(s,E)/\Omega(E)\]
where for any elliptic curve $A$ over $\BQ$, $\Omega(A)$ is the real period for the Neron differential of $A$; and for simplicity, we let $\Omega=\Omega(E)$; then the imaginary period of $E$ is $\Omega^-=\Omega/\sqrt{-3}$.

\begin{prop}\label{walds}
  Up to $\pm 1$, $L^\alg(1,E^{(M)}) = 2^{-1} P_{\chi_M}(f_M)$.
\end{prop}
\begin{proof}
By  Theorem  \ref{cst-wald},
$$L(1,E,\chi_M)=2^{-1}\frac{8\pi^2(\phi,\phi)_{\Gamma_0(36)}}{\sqrt{3}M}\frac{|P_{\chi_M}(f_M)|^2}{\langle f_M,f_M \rangle}.$$
Here,
\[\langle f_M,f_M \rangle = \frac{||f_M||^2}{2} \Vol(X_U)\]
and $\Vol(X_U)$ is the mass of $U$. By \cite[Lemma 2.2]{CST},
\[\Vol(X_U) = 2(4\pi^2)^{-1}\Vol(U)^{-1}\]
where $\Vol(U)$ is with respect to Tamagawa measures so that for any finite $p\not=3$,
$\Vol(\GL_2(\BZ_p)) = L(2,1_p)^{-1}$ and $\Vol(\CO_{B,3}^\times) = 2^{-1}L(2,1_3)^{-1}$.
By \cite[Lemma 3.5]{CST}, $\Vol(X_U) = 4/3$. Thus, $\deg_U f_M = 2/3$. On the other hand,
\[8\pi^2(\phi,\phi)_{\Gamma_0(36)} = 8\pi^2\int_{\Gamma_0(36) \bs \CH} |\phi(x+iy)|^2dxdy = i\Omega\Omega^-.\]
As $E^{(M)}$ and $E^{(-3M)}$ are isogenous over $\BQ$,
$L(s,E,\chi_M)=L(s,E^{(M)})L(s,E^{(-3M)}) = L(s,E^{(M)})^2$.
Denote by $\Omega^{(M)}$ the real period for $E^{(M)}$, then $\Omega^{(M)} = \Omega/\sqrt{M}$. Thus,
$L^\alg(1,E^{(M)})^2 = (L(1,E^{(M)})/\Omega^{(M)})^2 = ML(1,E,\chi_M)/\Omega^2$ and
$$L^\alg(1,E^{(M)})^2=2^{-2}|P_{\chi_M}(f_M)|^2.$$
\end{proof}

\subsection{Rank Zero Twists}
Keep the notations from the last section.
Denote by $\msa=\Gal(H_M/K)$, then $2\msa=\Gal(H_M/H_M^0)$, where $H_M^0=K(\sqrt{q}:q\mid M)$.
Let $\wh{\msa}$ (resp.$\wh{\msa/2\msa}$) be group of characters on $\msa$ (resp. on $\msa$ and factors through $\Gal(H_M^0/K)$).
Then
$$\sum\limits_{\chi\in\wh{\msa/2\msa}}P_\chi(f_M)=2^r y_0, \quad y_0 := \sum\limits_{\sigma\in2\msa}f_M(\sigma)$$
Note that each $\chi\in\wh{\msa/2\msa}$ corresponds to an integer $d|M$,
in the sense that $\chi$ corresponds to the extension $K(\sqrt{d})/K$.

\begin{prop}\label{prop:zero}
If $\chi\in\wh{\msa/2\msa}$ corresponds to an integer $d\neq M$, then $P_\chi(f_M)=0$
\end{prop}
\begin{proof}
Choose primes $q\in \BN_S$ such that $qd|M$. Then
$$\begin{array}{rl}P_\chi(f_M)& =\displaystyle{\sum_{\sigma\in\msa}f_M(\sigma)\chi(\sigma)}\\
 & \displaystyle{=\sum_{\sigma\in\Gal(H_{M/q}/K)}\sum_{\tau\in\Gal(H_M/H_{M/q})}f_M(\sigma\tau)\chi(\sigma\tau)}\\
 & \displaystyle{=\sum_{\sigma\in\Gal(H_{M/q}/K)}\chi(\sigma)\sum_{\tau\in\Gal(H_M/H_{M/q})}f(\sigma\tau x_M)}
\end{array}$$
By Theorem \ref{Nr:Gr}, we have
\[u_{M/q}\sum_{\tau\in\Gal(H_M/H_{M/q})}f(\sigma\tau x_M)=a_qf(\sigma x_{M/q})=0.\]
So the proposition holds.
\end{proof}

By Proposition \ref{prop:zero}, we have the equality
$$P_{\chi_M}(f_M)=2^ry_0$$

\begin{lem}
   The values of $f_M|_{\wh{B}^{\times 2}}$ are odd. In particular, $y_0$ is odd and
   \[v_2(P_{\chi_M}(f_M)) = r.\]
\end{lem}
\begin{proof}
   By the definition of $f_M$,
   $f_M|_{\wh{B}^{\times 2}}$ is odd if and only if for any $g \in \wh{B}^{\times 2}$,
   $\chi_0(g x_M) = (-1)^r$. Since $\chi_0$ is quadratic, $\chi_0(g) = 1$. Then
   $\chi_0(g x_M) = \chi_0(x_M) = \prod_{i=1}^r \chi_0(x_{q_i}) = (-1)^r$ as $q_i$ is
   inert in $K$. Hence
   \[y_0 \equiv [H_M: H_M^0] \equiv \frac{1}{3}\prod_{q|M} \frac{q+1}{2} \equiv 1 \pmod{2}.\]
\end{proof}

\begin{proof}[The Proof of Theorem \ref{th:zero}]
By Proposition \ref{walds}, $v_2(L^\alg(1,E^{(M)}))=r-1$.
The 2-part of $BSD$ is equivalent to
\[v_2\left(L^\alg\Big(1,E^{(M)}\Big)\right)=
\sum\limits_{p|6M}v_2\left(c_p\Big(E^{(M)}\Big)\right)-2v_2\left(\#E^{(M)}_\tor\right)+
v_2\left(\#\Sha\Big(E^{(M)}/\BQ\Big)\right).\]
The Tamagawa numbers of $E^{(M)}$ are: $c_2(E^{(M)})=3$(resp.  $=1$) if
$M \equiv 1 \pmod{8}$ (resp. otherwise), $c_3(E^{(M)}) = 2$ and $c_q(E^{(M)})=2$ for $q|M$.
On the other hand, $E^{(M)}(\BQ) = E^{(M)}(\BQ)_\tor = \BZ/2\BZ$.
Finally, using classical $2$-descent, $\Sha(E^{(M)}/\BQ)[2]=0$.
Combine the results above, it is clear that the $2$-part of BSD conjecture holds.

By \cite[Theorem 11.1]{Ru}, the $p$-part of the BSD-conjecture for $E^{(M)}$ holds for $p\nmid 6$,
hence and the first part of Theorem \ref{th:zero} holds.
\end{proof}

\subsection{The Gross-Zagier Formula}

Let $K = \BQ(\sqrt{-\ell})$ with $\ell \equiv 11 \pmod{12}$. Let $N = 36$. Write $N = N_0N_1^2$ as before.
There are two cases:
\begin{enumerate}
	\item if $\ell \equiv -1 \pmod{24}$, then the Heegner hypothesis holds and $N = N_0 = 36$;
	\item if $\ell \equiv 11 \pmod{24}$, then $N_0 = 9$.
\end{enumerate}
Embed $K$ into $M_2(\BQ)$ as $i_c$ with $c=1$ in Example \ref{example-embed}. Precisely,
 take an odd integer $a$ with $4\cdot N_0|(\ell + a^2)$
 and embed $K$ into $M_2(\BQ)$ by
  \[\sqrt{-\ell} \mapsto
  \begin{pmatrix}
    a & 2 \\
    -\frac{\ell+a^2}{2} & -a
\end{pmatrix}.\]
Then $M_2(\BZ) \cap K = R_0(N_0)\cap K = \CO_K$.
Under such embedding, take $R = \CO_K + N_1R_0(N_0)$ and consider the modular curve $X_K(N)$. For the
Heegner hypothesis case, $X_K(N) = X_0(36)$. For another case, the modular curve $X_K(N)$ has genus one
and by Lemma \ref{lem-cusp}, the cusp $[\infty]$ is defined over $\BQ$. In fact, by \cite[Example 11.7.c]{GR},
$A := (X_K(N),[\infty])$ is the elliptic curve
\[y^2 = x^3 - 27 \quad \text{(36C)}\]
which is $3$-isogenous to $E$. We have $A(\BQ) = A(\BQ)[2] \cong \BZ/2\BZ$.
For the Heegner hypothesis case (resp. the another case), take $f$ to be
the identity morphism on $E$ (resp. on $A$). Denote by
\[j =  \begin{pmatrix}
		1 & 0 \\
		-a & -1
\end{pmatrix} \in K^-.\]

\begin{lem}\label{key-lemma}
Take $w \in \GL_2(\wh{\BQ})$ the Atkin-Lehner operator defined in Section \ref{section-X_K(N)}. Precisely,
for the Heegner hypothesis case, $w = j^{(36)}w_{2}w_{3}$ while for another case, $w = j^{(3)}w_3$.
Then $w$ normalize $\wh{R}^\times$ and $w=t_0ju$ for some $t_0\in\wh{K}^\times$
and $u\in \wh{R}^\times$. Moreover, $f+f^w$ is a constant map and its image is not in $2E(\BQ)$ for the Heegner hypothesis
case or not in $2A(\BQ) = \{O\}$ for the another case.
\end{lem}
\begin{proof}
	By Lemma \ref{Bir}, it suffices to prove the ``Moreover'' part.

	For the Heegner hypothesis case,
	denote by $\Hom_{[\infty]}(X_0(N),E)$ the space of $\BQ$-morphisms from $X_0(N)$ to $E$
	taking $[\infty]$ to $O$ and $\Hom^0_{[\infty]}(X_0(N),E) = \Hom_{[\infty]}(X_0(N),E) \otimes_\BZ \BQ$.
	By Atkin-Lehner theory, $f^w=-f$ in $\Hom^0_{[\infty]}(X_0(N),E)$.
	So $f^w+f$ is  a constant map.
	However, $f([\infty])=O$, $f^w([\infty])=f([0]) = [0]$, while $[0]$ is the generator of $E(\BQ) = \BZ/6\BZ$.
	Thus, the image of $f + f^w$ is not in $2E(\BQ)$.

	For another case, view $f \in \Hom_{[\infty]}^0(X_K(N),A)$.
        Then $f^{w_3} = \epsilon(A/\BQ_3)f = f$.
  As $f$ and $f^{j_2}$ are both $K_2^\times$-invariant and such vectors in $\Hom_{[\infty]}^0(X_K(N),A)$
  is of dimension $1$. There is a sign $\epsilon \in \{\pm 1\}$ such that $f^{j_2} = \epsilon_2 f$.
  By  \cite[Theorem 4]{Pra},
  the sign $\epsilon_2 = +1$ if and only if $\epsilon(A/\BQ_2)  = \epsilon(A^{(-\ell)}/\BQ_2) = 1$.
  Since $\epsilon(A/\BQ_2) = -1$, we obtain $f^{j_2} = -f$. Thus $f^w = -f$ and
  as a morphism from $X_K(N)$ to $A$, $f+ f^w = T$ for some torsion point $T \in A(\BQ)$.
  To see $T \not= O$, it suffices to show $[\infty] \not= [\infty]^w$. This equivalent to
  say $w \not\in P(\BQ)\wh{R}^\times$ with $P$ the upper-triangler matrices in $\GL_2$. This holds
  since $w_3 \not\in P(\BQ)R_3^\times$.
\end{proof}

Write
$\Isom_\BQ(A)$ for the group of algebraic isomorphisms of
$A$ over $\BQ$ and $\Aut_\BQ(A)$ the subrgroup of algebraic isomorphisms
over $\BQ$ which fix $O$. Then $\Aut_\BQ(A) = \BZ/2\BZ$ is generated by
multiplication $-1$ and $\Isom_\BQ(A) = \langle t_T \rangle \times
\Aut_\BQ(A)$ where $t_T: P \mapsto P +T$ for any $P \in A$.

\begin{lem}
  For any $P \in A$, $P^{w_3} = t_T(P)$ and $P^{j^{(3)}} = -P$.
\end{lem}
\begin{proof}
  In the above proof, we have seen that $w_3 \not\in PU_3$. Thus $[\infty]^{w_3} \not= [\infty]$.
  Hence for any point $P$, $P^{w_3} = t_T(P)$. On the other hand, $P^{w} = t_T(-P)$. Therefore,
  $P^{j^{(3)}} = (P^w)^{w_3^{-1}} = -P$.
\end{proof}

Let $M = \prod_i q_i$ where $q_i$ are distinct positive integers $\equiv 5 \pmod{12}$.
Denoted by $\chi_M$ the quadratic character of $K$ associated to the extension $K(\sqrt{M})/K$.
Let $P_M \in X_K(N)(H_M)$ be
the Heegner point defined in \ref{euler}. Consider
\[P_{\chi_M}(f) = \sum_{\sigma \in \Gal(H_M/K)} f(P_M)^\sigma \chi_M(\sigma) \in E(K).\]

\begin{prop}\label{Gr1}
	Up to $\pm 1$,
	\[L^\alg(1,E^{(M)})\frac{L^{'}(1,E^{(-\ell M)})}{\Omega(E^{(-\ell M)})} = \wh{h}_K(P_{\chi_M}(f)).\]
\end{prop}
\begin{proof}
By  Theorem \ref{GZ},
$$L'(1,E,\chi_M)=\frac{8\pi^2(\phi,\phi)_{\Gamma_0(36)}}{\sqrt{\ell}M}\cdot\wh{h}_K(P_{\chi_M}(f)).$$
Since $L(s,E,\chi_M)=L(s,E^{(M)})L(s,E^{(-\ell M)})$, and we have proved that $L(s,E^{(M)})$ is nonvanishing at $s=1$,
$$L'(1,E,\chi_M)=L(1,E^{(M)})L'(1,E^{(-\ell M)}).$$
As in in the proof of \ref{walds}
\[8\pi^2(\phi,\phi)_{\Gamma_0(36)} = 8\pi^2\int_{\Gamma_0(36) \bs \CH} |\phi(x+iy)|^2dxdy = i\Omega\Omega^-.\]
By \cite{VP}, we know $\Omega(E^{(M)})=\Omega/\sqrt{M}$ and up to sign
$\Omega(E^{(-\ell M)})=\Omega^-/\sqrt{-\ell M}$, so up to sign
$$\Omega(E^{(M)})\Omega(E^{(-\ell M)})=\frac{\Omega}{\sqrt{M}}\frac{\Omega^-}{\sqrt{-\ell M}}
=-\frac{8\pi^2(\phi,\phi)_{\Gamma_0(36)}}{M\sqrt{\ell}}$$
Thus up to sign:
$$L^\alg(1,E^{(M)})\frac{L^{'}(1,E^{(-\ell M)})}{\Omega(E^{(-\ell M)})}=\wh{h}_K(P_{\chi_M}(f))$$
\end{proof}

\subsection{Rank One Twists}

Let $\ell$ be a prime with $\ell \equiv 11 \pmod{12}$. Denote by $K = \BQ(\sqrt{-\ell})$.
We only prove Theorem \ref{th:one} in the case $\ell \equiv 11 \pmod{24}$,  that is,
$2$ is inert in $K$ and $3$ is split in $K$, while its proof for the other case is similar.

Let $M = q_1 \cdots q_r$ where $q_i$ are distinct primes such that $q_i \equiv 5 \pmod{12}$ and
inert in $K$. Denote by $\msa=\Gal(H_M/K)$. Then $2\msa=\Gal(H_M/H_M^0)$,
where $H_M^0=K(\sqrt{q}:q\mid M)$. Let $\wh{\msa}$ (resp.$\wh{\msa/2\msa}$) be the group of characters on $\msa$
(resp. on $\msa$ and factors through $\Gal(H_M^0/K)$).

Let $A$ be the elliptic curve $y^2=x^3-27$. Observe that $A(H_M^0)[2^\infty]=A(\BQ)[2^\infty]=A(\BQ)[2]$. In fact, suppose
$Q\in A(H_M^0)[2^\infty]$ but $Q\notin A(\BQ)[2^\infty]$. Then the extension $\BQ(Q)/\BQ$  is unramified outside $2$ and $3$.
However, as $\BQ(Q) \subset H_M^0$, $\BQ(Q)/\BQ$ must be ramified at $\ell$ or $q_i$ for some $i$. A contradiction.
Let $T$ be the nontrivial element in $A(\BQ)[2]$, and $C=\#A(H_M^0)_\tor/2$ be the cardinality of odd part of $A(H_M^0)_\tor$.
Denote by
\[y_M = P_{\chi_M}(f) = \sum_{\sigma \in \msa} f(P_M)^{\sigma}\chi_M(\sigma) \in A(H_M^0).\]
Then the same as Proposition \ref{prop:zero}, we have
$$y_M=2^ry_0, \quad y_0 := \sum\limits_{\sigma\in 2\msa}f(P_M)^{\sigma},$$
as equality of points in $A(H_M^0)$. The key point is the following lemma:
\begin{lem}\label{lem:key}
$$\ov{y}_0+y_0=T$$
\end{lem}
\begin{proof}
By Lemma \ref{Bir}, one can write $w = t_0ju$ with $t_0 \in \wh{K}^\times$, $j = K^{-}$ and
$u \in \wh{R}^\times$. Take $x_M \in \wh{B}^\times$ such that
$P_M = [z,x_M] \in X_K(N)(H_M)$ with $z \in \CH^{K^\times}$.
Thus, for any $\sigma_t \in 2\msa$ with $t \in \wh{K}^\times$
\[f^w(P_M)^{\sigma_t} = f([h_0,tx_Mt_0j]).\]
Note that $x_M \in \GL_2(\BQ^{(N)})$ while $t_0 \in K_{(N)}^\times \subset \GL_2(\BQ_{(N)})$.
Hence $x_Mt_0 = t_0x_M$ and
\[f^w(P_M)^{\sigma_t} = f([h_0,x_Mj])^{\sigma_{tt_0}}.\]
Finally, we need to show that $x_Mj \in jx_MU$. This reduces to show that for any $q|M$,
the $q$-part of $x_M^{-1}j^{-1}x_Mj$ belongs to $R_q^\times = \GL_2(\BZ_q)$.  It is easy to
check this holds. Thus
\[f^w(P_M)^{\sigma_t} = f([h_0,jx_M])^{\sigma_{tt_0}} = \ov{f([h_0,x_M])}^{\sigma_{tt_0}}.\]

On the other hand, note that in the proof of Lemma \ref{Bir}, $t_{0,p} = 1$ for any $p \nmid N_0 = 9$.
Denote by $a_0 = N_{K/\BQ}(t_0) \in \wh{\BQ}^\times$. Take determinant for the equation $w = jt_0u$.
Then $a_{0,p} = 1$ if $p \not= 3$ and $a_{0,3} \in 9\BZ_3^\times$. Thus for any prime $q|M$
\[\sigma_{t_0}(\sqrt{q}) = \sigma_{a_0}(\sqrt{q}) = \sqrt{q}\]
where $\sigma_{a_0} \in \Gal(\BQ(\sqrt{q})/\BQ)$ via the Artin map over $\BQ$. Hence, $\sigma_{t_0} \in 2\CA$.

Sum up, since $\displaystyle{[H_M:H_M^0]=[H:K]\prod_{q|M}\frac{q+1}{2}}$ is odd, we get
$$\ov{y}_0+y_0=\sum_{\sigma\in2\msa}(f+f^w)(P_M)^{\sigma}=[H_M:H_M^0]T=T.$$
\end{proof}

\begin{thm}\label{thm:in}
$y_M\in A(K(\sqrt{M}))^-$ and the 2-index of $y_M$ is $r-1$ in $A(K(\sqrt{M}))$.
\end{thm}
\begin{proof}
Consider the maps
$$\xymatrix{
        &                                          & A(K(\sqrt{M}))/2^rA(K(\sqrt{M}))\ar[d]^\delta & \\
0\ar[r] & H^1(H_M^0/K(\sqrt{M}),A[2^r](H_M^0))\ar[r] & H^1(K(\sqrt{M}),A[2^r])\ar[r] & H^1(H_M^0,A[2^r])
}$$
where $\delta$ is the Kummer map, which is injective, and the horizontal line is the inflation-restriction exact sequence.
Since $y_M=2^ry_0$ with $y_0\in A(H_M^0)$, the image of $\delta(y_M)$ is $0$ in $H^1(H_M^0,A[2^r])$,
hence $\delta(y_M)$ lies in the image of $H^1(H_M^0/K(\sqrt{M}),A[2^r](H_M^0))$,
which is killed by $2$. It follows that $2y_M\in 2^rA(K(\sqrt{M}))$, then
$$y_M=2^{r-1}z+t,z=2y_0+s$$
for some $z\in A(K(\sqrt{M}))$ and $s,t\in A(\BQ)[2]$.

Let $\sigma\in\Gal(K(\sqrt{M})/K)$ be the nontrivial element. Then by definition,
we have $y_M+y_M^\sigma=0$, so $y_0+y_0^\sigma\in A(H_M^0)[2^r]=A(\BQ)[2]$, so $z+z^\sigma=0$. On the other hand
$$z+\ov{z}=2(y_0+\ov{y}_0)=0,$$
which implies $z\in A(\BQ(\sqrt{-\ell M}))^-=A^{(-\ell M)}(\BQ)$. Therefore
$$y_M\in 2^{r-1}A(\BQ(\sqrt{-\ell M}))^-+A(\BQ)[2].$$

We will show that the $2$-index of $y_M$ is exactly $r-1$. Suppose that $y_M=2^rz+t$ for some $z\in A(\BQ(\sqrt{-\ell M}))^-$
and $t\in A(\BQ(\sqrt{-\ell M}))_\tor$. Then $2^r(z-y_0)+t=0$. which implies $C(z-y_0)\in A(\BQ)[2]$.
Operating by complex conjugation and plus together, $C(z-y_0)+C(\ov{z}-\ov{y}_0)=0$. But we have $z+\ov{z}=0$,
so $C(y_0+\ov{y}_0)=0$. But it contradicts to the fact that $\ov{y}_0+y_0=T\neq 0$.
\end{proof}

\begin{proof}[The Proof of Theorem \ref{th:one}]
Observe that $A$ and $E$ are $3$-isogeny, so to prove Theorem \ref{th:one}, we only need to prove that it is holds for $A$.

By Proposition \ref{Gr1}, up to $\pm 1$,
\[L^\alg(1,A^{(M)})\frac{L^{'}(1,A^{(-\ell M)})}{\Omega(A^{(-\ell M)})} = \wh{h}_K(y_M).\]
Denote by $R(-\ell M) = \wh{h}(P_{-\ell M})$ where $P_{-\ell M}$ is the generator of
$A^{(-\ell M)}(\BQ)/A^{(-\ell M)}(\BQ)_\tor$. In particular, by Theorem \ref{thm:in}
\[\wh{h}_K(y_M) = 2^{2r-1}R(-\ell M)\]
Thus, if denote by
\[L^{'\alg}(s,A^{(-\ell M)}) = \frac{L^{'}(s,A^{(-\ell M)})}{R(-\ell M)\Omega(A^{(-\ell M)})}\]
then by the result of rank zero case, we have
$$v_2(L^{'\alg}(1,A^{(-\ell M)}))=r.$$

The Tamagawa numbers of $A^{(-\ell M)}$ are: $c_2(A^{(-\ell M)})=1\text{ or }3$, $c_3(A^{(-\ell M)}) = 2$
and $c_q(A^{(-\ell M)})=2$ for $q|\ell M$.
On the other hand, $A^{(-\ell M)}(\BQ) = A^{(-\ell M)}(\BQ)_\tor = \BZ/2\BZ$.
Finally, using classical $2$-descent, $\Sha(A^{(-\ell M)}/\BQ)[2]=0$.
Combine the results above, it is clear that the $2$-part of BSD conjecture for $A^{(-\ell M)}$ holds.

Since $A^{(-\ell M)}$ has CM, so the $p$-adic height paring is nontrivial.
Then by \cite[Corollary 1.9]{PR}, $p$-part of the BSD conjecture for $A^{(-\ell M)}$ holds for $p\nmid 6\ell M$.
So the second part of Theorem \ref{th:one} holds for $A$, and hence for $E$.
\end{proof}

\end{document}